\documentclass{article}
\usepackage{graphicx}
\usepackage
[
colorlinks=true,
linkcolor=blue,
anchorcolor=blue,
citecolor=blue,
urlcolor=blue,
plainpages=false,
pdfpagelabels
]{hyperref}

\usepackage{bbm,amssymb}
\usepackage{amsmath}
\usepackage{amsthm,thmtools}
\usepackage[marginpar]{todo}
\usepackage[numbers]{natbib}
\usepackage{xcolor}
\theoremstyle{plain}
\newtheorem{theorem}{Theorem}[section]

\newtheorem{lemma}[theorem]{Lemma}
\newtheorem{corollary}[theorem]{Corollary}

\theoremstyle{definition}
\newtheorem{definition}[theorem]{Definition}
\declaretheorem[name=Example,numbered=no,style=definition,qed={\lower-0.3ex\hbox{$\triangle$}}]{example}

\declaretheorem[name=Remark,sibling=theorem,style=remark,qed={\lower-0.3ex\hbox{$\Diamond$}}]{remark}

\DeclareMathOperator{\sech}{sech}

\newcommand{\Z}{\mathbbm Z}
\newcommand{\R}{\mathbbm R}

\newcommand{\Stwo}{\mathbbm S^2}
\newcommand{\spann}{\mathop{\mathrm{span}}}

\newcommand{\qOne}{\ensuremath{\mathbbm{1}}}

\newcommand{\Quat}{\mathbbm{H}}

\newcommand{\quadmatrix}[4]{\left(\begin{array}{cc}#1&#2\\#3&#4\end{array}\right)}
\newcommand{\qIm}[1]{ \left[ #1 \right]^{\mathrm{tr} = 0}}

\author{Tim Hoffmann\thanks{This research was supported by the DFG-Collaborative Research Center, TRR 109, ``Discretization in Geometry and Dynamics.''} \, and \,Andrew O. Sageman-Furnas}
\title{A $2\times 2$ Lax representation, associated family, and B\"acklund transformation for circular K-nets}

\begin{document}
\maketitle

\begin{figure}[h]
  \centering
  \includegraphics[width=.7\hsize]{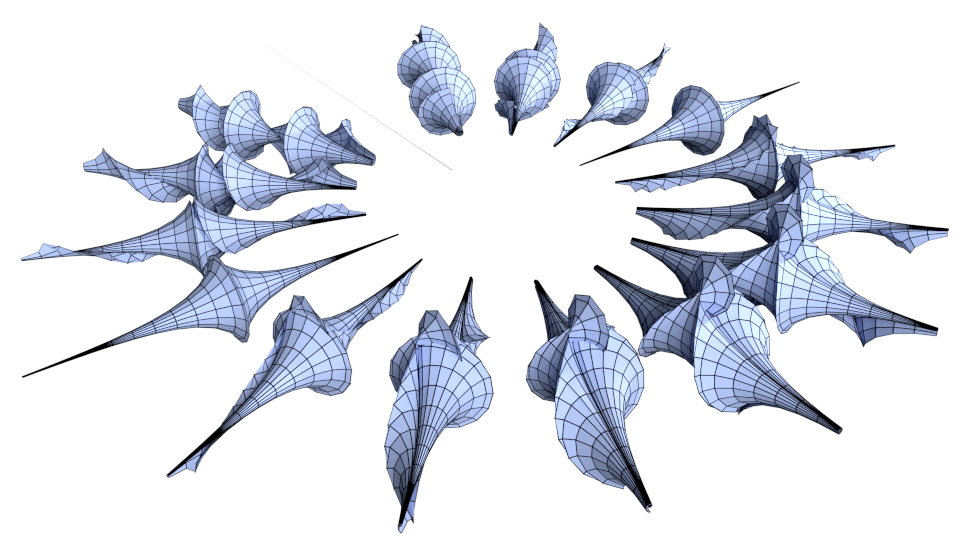}
  \caption{The B\"acklund transformations of the discrete curvature line Pseudosphere, aligned by angular parameter.}
  \label{fig:clock}
\end{figure}

\begin{abstract}
We present a $2\times 2$ Lax representation for discrete circular nets of constant negative Gau{\ss} curvature. It is tightly linked to the 4D consistency of the Lax representation of discrete K-nets (in asymptotic line parametrization). The description gives rise to B\"acklund transformations and an associated family. All the members of that family -- although no longer circular -- can be shown to have constant Gau{\ss} curvature as well. Explicit solutions for the B\"acklund transformations of the vacuum (in particular Dini's surfaces and breather solutions) and their respective associated families are given.
\end{abstract}

\section{Introduction}
\label{sec:intro}
Smooth surfaces of constant negative curvature and their transformations are a classical topic of differential geometry (for a modern treatment see, e.g., the book by Rogers and Schief \cite{Rogers:2002wy}). 

Discrete analogues of surfaces of constant negative Gau{\ss} curvature in \emph{asymptotic parametrization} (now known as \emph{K-nets}) and their B\"acklund transformations were originally defined by Wunderlich \cite{Wunderlich:1951wc} and Sauer \cite{Sauer:1950ca} in the early 1950s. In 1996 Bobenko and Pinkall \cite{Bobenko:1996ug} showed that these geometrically defined K-nets are equivalent to ones arising algebraically from a discrete moving frame 2x2 Lax representation of the well known discrete Hirota equation \cite{Hirota:1977cj}. This algebraic viewpoint highlights the interrelationship between a discrete net, its associated family (generated by the spectral parameter of the Lax representation, which corresponds to reparametrization of the asymptotic lines), and its B\"acklund transformations (arising from the 3D consistency of the underlying discrete evolution equation).

In the smooth setting, surface reparametrization is a simple change of variables and does not affect the underlying geometry. However, understanding surface reparam\-etrization in the discrete setting is a much more delicate issue. In particular, discrete analogues of constant negative Gau{\ss} curvature surfaces in \emph{curvature line parametrizations} have been defined and studied by restricting a notion of discrete curvature line parametrization (called \emph{C-nets} since each quad is concircular) with its corresponding definition of Gau{\ss} curvature \cite{Konopelchenko:1999te,Schief20061484,Bobenko:2010eg}. We call such objects \emph{circular K-nets} or \emph{cK-nets} and will be our main focus. Recently, a curvature theory has been introduced for a more general class of nets (so-called \emph{edge-constraint} nets) that furnishes both asymptotic K-nets and cK-nets with constant negative Gau{\ss} curvature \cite{Hoffmann:2014wq}.

In \cite{Schief:2003ug} Schief gave a Lax representation for circular K-nets in terms of $3\times 3$ matrices in the framework of a special reduction of C-nets and showed how circular 3D compatibility cubes give rise to B\"acklund transformations.  However, this B\"acklund transformation corresponds to a double B\"acklund transformation in smooth setting and the relationship between cK-nets and asymptotic K-nets remained unclear.

In what follows we show that cK-nets, discrete curvature line nets of constant negative Gau{\ss} curvature, exhibit a $2\times2$ Lax pair, associated family, and B\"acklund transformations that naturally arise from their construction as the diagonals of asymptotic K-net quadrilaterals with all edge lengths equal. In other words a cK-net Lax matrix is the product of two K-net matrices. This is reasonable and expected since curvature coordinates are the sum and difference of asymptotic ones. However, we wish to highlight three important subtleties that arise:

\begin{enumerate}
\item There are more cK-nets than those given by connecting the diagonals of K-nets with equal side lengths and retopologizing: for example, as shown in Figure \ref{fig:eightLoop}, even though each edge factors into the diagonal of a K-net quadrilateral, the four corresponding K-net quads do not share a central vertex.
\item The associated family of cK-nets yields nets in more general parametrization (since the spectral parameter corresponds to reparametrization of the asymptotic lines), but as shown in Theorem \ref{thm:ckFromLax} they are all constant negative Gau{\ss} curvature edge-constraint nets.
\item As shown in Figure \ref{fig:cube}, the 3D compatibility cube corresponding to the B\"acklund transformation of cK-nets is unusual since the equations for its sides are not the same as that of its top and bottom. However, double B\"acklund transformations with negative angular parameters do form a usual 3D consistent cube with circular faces. These double B\"acklund transformations also accept complex angular parameters, yielding nonfactorizable breather surfaces, as shown in Figure \ref{fig:complexBT}.
\end{enumerate}

The paper is organized as follows: after introducing some preliminaries, we recapitulate many facts about asymptotic K-nets. Then we briefly review the recently introduced theory of edge-constraint nets and their curvatures. The main results are in Section \ref{sec:cK-nets}. In Section \ref{sec:Lax} we define the Lax matrices and prove they give rise to edge-constraint nets. In \ref{sec:cKGeometry} we see that these are in fact Lax matrices for cK-nets (and their associated families) and that every cK-net arises in this way. The B\"acklund transformation for cK-nets is given in Section \ref{sec:BT}. Finally, in Section \ref{sec:lineSingleBT} we present closed form equations for some B\"acklund transformations of the straight line, yielding, e.g., discrete analogues of Dini's surfaces, Kuen's surface, and breather surfaces, together with their respective associated families.

\subsection{Preliminaries and notation}
We consider a discrete analogue of parametrized surfaces in $\R^3$ known as quad nets.
\begin{definition}
  A \emph{quad graph} $G$ is a strongly regular polytopal cell
  decomposition of a regular surface with all faces being quadrilaterals.
  A \emph{quad net} is an immersion of a quad graph into $\R^3$.
\end{definition}
For simplicity we assume $G$ to be $\Z^2$ in the following sections, though all results generalize to \emph{edge-bipartite} quad graphs. \footnote{For more general edge-bipartite graphs, vertices with valence greater than four might not have a continuous limit in the classical sense. For example, if the quad net is a discrete constant negative Gau{\ss} curvature surface parametrized by curvature lines, then these points are something like "Lorentz umbilics" \cite{Dorfmeister:2009eh}.} Furthermore, we will associate a unit normal to each vertex of a quad net, equipping it with a discrete Gau{\ss} map $n: G \to \Stwo$. This will be further explained in Section \ref{sec:steiner}.

To distinguish arbitrary vertices of a quad net (or its Gau{\ss} map) we will use shift notation. For $f:\Z^2\to\R^3$, $f$ will denote the map at a vertex $(k,\ell)$ and subsequent subindices will stand for shifts in the corresponding lattice directions: $f=f(k,\ell)$, $f_1:= f(k+1,\ell)$, $f_2:= f(k,\ell+1)$, $f_{12} = f(k+1,\ell+1)$, etc.

Discrete integrable surface theory has well established analogues of asymptotic (A-net) and curvature line (C-net) parametrizations \cite{Bobenko:2008tn}.
\begin{definition}
An \emph{A-net} is a quad net where each vertex star lies in a plane.
\end{definition}
\begin{definition}
A \emph{C-net} is a quad net where each face is inscribed in a circle.
\end{definition}

\subsection{(asymptotic) K-nets}
\label{sec:K-nets}
The theory of K-nets -- discretizations of surfaces of constant negative Gau{\ss} curvature in asymptotic line parametrization -- is well established (see, for example, \cite{Wunderlich:1951wc,Sauer:1950ca,Bobenko:1996ug,Hoffmann:1999wq,pinkall2008designing}). Geometrically, K-nets are A-nets in which every quad is a skew parallelogram, though we will equivalently define them using their moving frame description. We briefly recapitulate this construction and other facts (reviewed in the book \cite{Bobenko:2008tn}) that we will need later.

We express the quaternions $\Quat$ in terms of 2x2 complex matrices as a real vector space over the Pauli matrices $\{\qOne, -i \sigma_1, -i \sigma_2, -i \sigma_3\}$, where
\begin{equation}
\label{eq:PauliMatrices}
\qOne = \quadmatrix1001, ~ \sigma_1 = \quadmatrix0110, ~ \sigma_2 = \quadmatrix0{-i}{i}0, ~ \sigma_3 = \quadmatrix100{-1},
\end{equation}
and identify $\R^3$ with the space of imaginary quaternions. We will denote the projection $\Quat \to \R^3$ induced by taking the quaternionic imaginary part of $q \in \Quat$ by $\left[ q \right]^{\mathrm{tr} = 0}$ since it corresponds to the trace free part in the 2x2 matrix representation.

Following \cite{Bobenko:1996ug} consider the quaternionic matrices (more precisely maps $U,V:\Z^2 \to \Quat$) given by 
\begin{equation}
  \label{eq:KLaxPair}
\begin{aligned}
    U &=\quadmatrix{\cot(\frac{\delta_u}2)\frac {H_1}{H}}{i \lambda}{i \lambda}{\cot(\frac{\delta_u}2)\frac {H}{H_1}}\\
    V &= \quadmatrix{1}{\frac{i}{\lambda}\tan(\frac{\delta_v}2)H_2H}{\frac{i}{\lambda}\tan(\frac{\delta_v}2)\frac1{H_2H}}{1}    
\end{aligned}
\end{equation}
depending on the so-called \emph{spectral parameter} $\lambda$ with $H = e^{i h}$ and the matrix problem
\begin{equation}
  \label{eq:ZeroCurvature}
\Phi_1 = U \Phi, \quad \Phi_2 = V \Phi.
\end{equation}
Under the assumption that $\delta_u$ and $\delta_v$ only depend on the second and first lattice directions, respectively, the integrability condition $V_1U = U_2V$ implies that the $h$ variables solve the Hirota equation \cite{Hirota:1977cj}
\begin{equation}
e^{i(h_{12}+h)} - e^{i(h_1+h_2)} =\tan\frac{\delta_u}2\tan\frac{\delta_v}2 \left( 1-e^{i(h+h_1+h_{12}+h_2)} \right).
\end{equation}
A quad net $f$ together with a Gau{\ss} map $n$ can then be generated for each $t \in \R$ via the following formulas:
\begin{equation}
  \label{eq:KSym}
  f = 2 \left[\Phi^{-1}\frac{\partial}{\partial t}\Phi\right]^{\mathrm{tr}=0}, \quad n = -i \Phi^{-1}\sigma_3 \Phi, \quad \lambda = e^t.
\end{equation}
This method of getting the immersion by differentiating with respect to the spectral parameter instead of integrating the frame is called the Sym \cite{Sym:1985kl} or Sym-Bobenko \cite{Bobenko:1994tv} formula.

\begin{definition}
For each $t \in \R$, the quad net $f$ arising from the system \eqref{eq:KLaxPair} by means of the Sym formula \eqref{eq:KSym} is called a {\em K-net}. The family of K-nets for all values of $t$ is called the {\em associated family} of each of its members.
\end{definition}

A K-net immersion $f$ can also be reconstructed (up to global scaling) from its Gau{\ss} map $n$ by the relations $f_1-f = n_1\times n$ and $f_2-f = n\times n_2$. The Gau{\ss} map $n$ solves the discrete Moutard equation restricted to $\Stwo$ (see \cite{Nimmo:1997dg}). The discrete Moutard equation is known to be 3D compatible, giving rise to a discrete version of the classical B\"acklund transformation and a corresponding permutability theorem (an alternative algebraic proof in terms of the Hirota equation is given in \cite{Bobenko:1996ug} and for more geometric insight see \cite{Sauer:1950ca,Wunderlich:1951wc}).

\begin{definition}
\label{def:knetBT}
Given a K-net $f:\Z^2\to\R^3$ with vertex normals $n$, an angle $\alpha \neq k \pi \in \R$, and a direction $v\perp n_{0,0}$ there exists a unique K-net $\hat f$ (with vertex normals $\hat n$) such that $(\hat f_{0,0}-f_{0,0})$ is parallel to $v$, $\Vert \hat f- f\Vert = \sin\alpha$, and $\angle(n,\hat n) = \alpha$. The resulting K-net $\hat f$ is called a {\em B\"acklund transform} of $f$.
\end{definition}
\begin{theorem}
  Consider a K-net $f:\Z^2\to \R^3$ with Gau\ss{} map $n$
  together two B\"acklund transforms $\hat f$ and $\tilde f$ with
  parameters $\hat \alpha$ and $\tilde \alpha$, respectively. Then there is a
  unique K-net $\hat{\tilde f}$ that is a
  $\tilde\alpha$-B\"acklund transform of $\hat f$ as well as a
  $\hat \alpha$-B\"acklund transform of $\tilde f$.
\end{theorem}

\subsection{Edge-constraint nets and curvatures}
\label{sec:steiner}
Let us briefly recall the notion of edge-constraint nets and their curvatures. The definition of edge-constraint nets is first and foremost a weak coupling of a quad net with its Gau\ss{} map. This allows for a curvature theory based on normal offsets that turns out to be consistent with many known discretizations of integrable surfaces. In the case of C-nets it coincides with the definitions given in \cite{Schief20061484,Bobenko:2010eg} which include the nets of constant mean curvature \cite{Bobenko:1999us} and minimal nets \cite{Bobenko:1996vq} defined by Bobenko and Pinkall. Even nets of constant negative Gau\ss{} curvature in asymptotic line (K-nets) and curvature line parametrization (cK-nets, the topic of this paper) are in this class. Moreover, the class of edge-constraint nets also includes the associated families of all of these constant curvature nets (and furnishes them with the expected curvatures). These concepts are discussed in depth in \cite{Hoffmann:2014wq}.

\begin{definition}
Let $G$ be a quad graph. Two maps $f:G\to \R^3$ and $n:G\to \Stwo$ are said to form an \emph{edge-constraint net} if
\begin{equation}f_i-f \perp n_i+n\end{equation} holds for all edges of the graph $G$. The map $n$ is then called a \emph{Gau\ss{} map} for $f$. A unit vector $N\perp \spann\{n_{12}-n,n_2-n_1\}$ is said to be a face normal. 
The Gau\ss{} and mean curvature for an edge-constraint net $f,n$ with face normal $N$ are given by
\begin{equation}
    \label{eq:GaussCurvature}
    K := \frac{\det(n_{12}-n,n_2-n_1,N)}{\det(f_{12}-f,f_2-f_1,N)}
  \end{equation}
  and 
  \begin{equation}
    \label{eq:MeanCurvature}
    H := \frac12\frac{\det(f_{12}-f,n_2-n_1,N)+\det(n_{12}-n,f_2-f_1,N)}{\det(f_{12}-f,f_2-f_1,N)},
  \end{equation}
  respectively.
\end{definition}

\begin{remark}
Generically the face normal is unique up to sign but even if it is not the above defined curvatures are invariant under the choice of $N$ (see \cite{Hoffmann:2014wq} for more details).
\end{remark}
  
\begin{remark}
This notion of curvature is motivated by the Steiner formula that relates the area of offset surfaces with the curvatures of the original one. If $f^t := f+ tn$ taken pointwise defines the offset surface, one finds $A(f^t) = (1 + 2 H t+ K t^2) A(f)$ where $A$ denotes the area of the surface over a given region and $H$ and $K$ are the integrals of the mean and Gau\ss{} curvature of $f$ over that region.
\end{remark}

\begin{lemma}
A K-net (with spectral parameter $\lambda = e^t$) is an edge-constraint net with Gau{\ss} curvature per quad given by
\begin{equation}
K = -2 \cosh^2t \frac{(1 - \cos\delta_u \tanh t)(1 + \cos\delta_v \tanh t)}{\cos\delta_u + \cos\delta_v}
\end{equation}
\begin{proof}
K-nets are edge-constraint since by construction $n \perp f_i - f$ holds for all edges incident with a given vertex. The Gau{\ss} curvature can be computed directly (for further details see again \cite{Hoffmann:2014wq}).
\end{proof}
\end{lemma}

\section{Circular K-nets}
\label{sec:cK-nets}
Circular nets of constant negative Gau\ss{} curvature have been discussed in \cite{Konopelchenko:1999te,Schief20061484,Bobenko:2010eg} by looking at C-nets and requiring that their Gau{\ss} curvature \eqref{eq:GaussCurvature} be $K = -\frac1{\rho^2}$ for some constant $\rho\neq0$. For lack of a better name we call such nets \emph{cK-nets}. Let us start with an example.
\begin{example}[Pseudosphere]
There is a natural discrete version of the tractrix construction as the curve halfway between a regular planar curve $\gamma$ and its \emph{Darboux transform} $\hat \gamma$, as shown in Figure~\ref{fig:DiscreteTractrix} left. Given a regular discrete curve $p$ (i.e., a polygon with no vanishing edges) and a starting point $\hat p_0$ at distance $2d>0$ from $p_0$, there is a unique polygon $\hat p$ such that: (i) $\Vert \hat p - p\Vert = 2d$, (ii) $\Vert \hat p_1 -\hat p\Vert = \Vert p_1 -p\Vert$, and (iii) the quadrilaterals $p,p_1,\hat p_1, \hat p$ form planar non-embedded parallelograms (parallelograms folded along their diagonals). The regular discrete curve $\hat p$ is known as the \emph{discrete Darboux transform} of $p$ \cite{Hoffmann:2008ub}. The \emph{tractrix $\tilde p$ of $p$} is defined as the polygon pointwise halfway between the two: $\tilde p = \frac12(\hat p + p)$. Normals to the smooth tractrix $\tilde \gamma$ are given (maybe up to sign) by normalizing the tangent vector $\hat \gamma - \gamma$ and rotating by 90 degrees. Similarly, we furnish the discrete tractrix $\tilde p$ with normals $\tilde n$ at vertices by taking 90 degree rotations of $\frac{\hat p - p}{\| \hat p - p \|}$, as shown in Figure~\ref{fig:DiscreteTractrix} right.

\begin{figure}[t]
\centering
\includegraphics[height=1.4in]{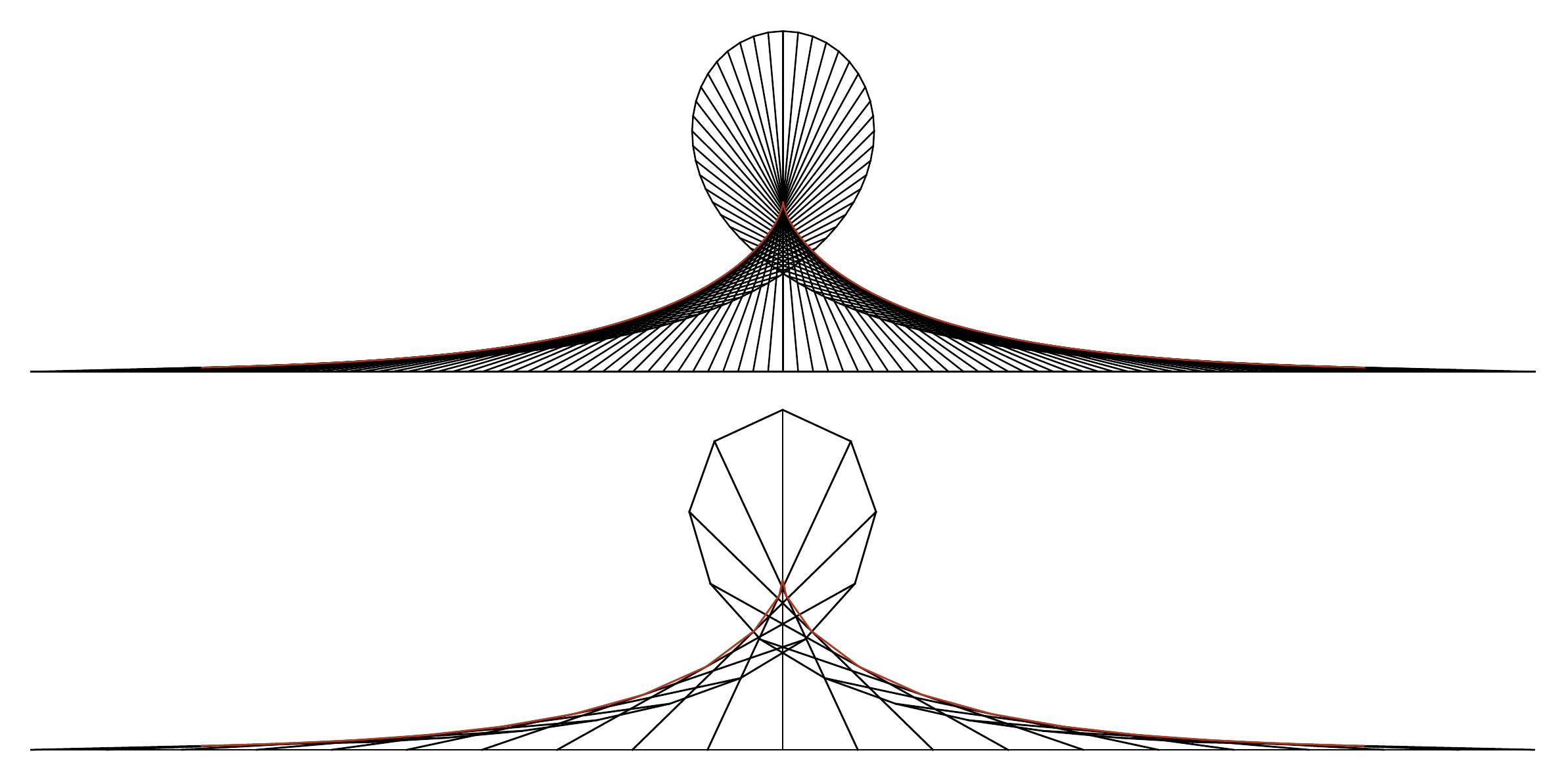}\quad\quad
\includegraphics[height=1.4in]{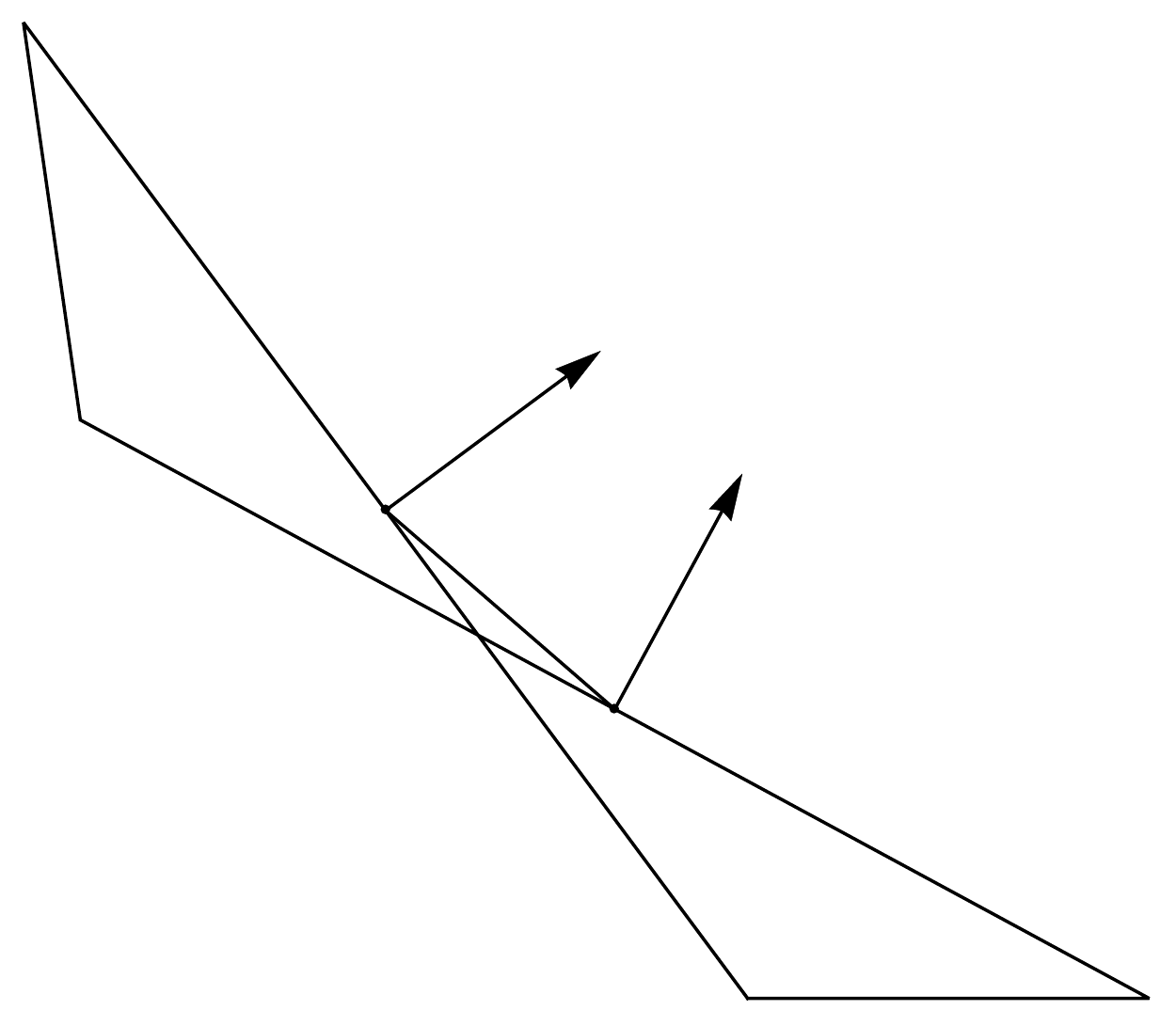}
  \caption{Left: Smooth (above) and discrete (below) Darboux transform and tractrix (red) of the straight line. Right: Natural normals for the discrete tractrix.}
  \label{fig:DiscreteTractrix}
\end{figure}

Starting from the polygon $p_k = (\epsilon k,0)$ and an initial point at a given distance, say $\hat p_0 = (0,2d)$ (this is the symmetric choice but that is not necessary), we generate a tractrix polygon $\tilde p$ together with normals $\tilde n$ that can be used to form a discrete surface of revolution: Given a rotation angle $\phi$ (choosing an integer fraction of $2\pi$ guarantees it closes in the rotational direction), define
\begin{equation}
f(k,\ell) =
\begin{bmatrix}
(\tilde p_k)_x\\
\cos(\ell \phi) (\tilde p_k)_y\\
\sin(\ell \phi)(\tilde p_k)_y
\end{bmatrix}
\end{equation}
and rotate the normals $\tilde n$ along with their corresponding points to form the Gau{\ss} map $n(k,\ell)$. By construction the quads of both $f$ and $n$ are planar isosceles trapezoids lying in parallel planes, so $(f,n)$ is a circular edge-constraint net. Through elementary geometry one can compute the signed area of each isosceles trapezoid of $f$ and its corresponding trapezoid of $n$. Since the face normal per quad is in fact perpendicular to each of these trapezoids, the ratio of these areas is the Gau{\ss} curvature \eqref{eq:GaussCurvature}, which is found to be $-\frac{1}{d^2}$ for every quad. In particular, it is independent of both discretization parameters $\epsilon$ and $\phi$, so all resulting discrete pseudospheres are cK-nets. 

Figure~\ref{fig:tractrixPseudosphere} shows a resulting discrete Pseudosphere.
\end{example}

\begin{figure}[t]
  \centering
  \includegraphics[width=.6\hsize]{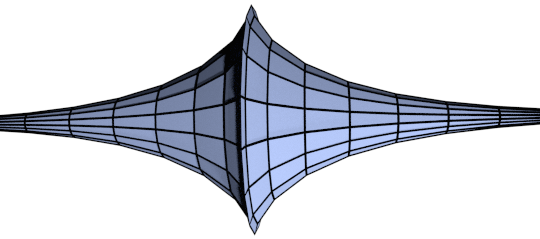}
  \caption{A cK-net pseudosphere of revolution.}
  \label{fig:tractrixPseudosphere}
\end{figure}

\begin{remark}
\label{rem:pseudoFormulas}
Schief \cite{Schief:2003ug} constructed by other methods the same pseudospheres as those above and provided an explicit formula for the discrete immersion. In \cite{Bobenko:2010eg}, Bobenko, Pottmann, and Wallner, gave an implicit relation for the meridian polygon and its normals to produce discrete cK-nets of revolution. The above tractrix construction provides explicit normals at vertices furnishing an edge-constraint net of constant negative Gau{\ss} curvature:
\begin{equation}
f(k, \ell) = 
\begin{bmatrix}
\epsilon k - \tanh(\tau k) \\
\cos(\ell \phi) \sech(\tau k) \\
\sin(\ell \phi) \sech(\tau k)
\end{bmatrix} \quad \mathrm{with} \quad  
n(k, \ell) = 
\begin{bmatrix}
\sech(\tau k) \\
\cos(\ell \phi) \tanh(\tau k) \\
\sin(\ell \phi) \tanh(\tau k)
\end{bmatrix},
\end{equation}
where $\tau = \log\frac{2+\epsilon}{2-\epsilon}$.
\end{remark}

\subsection{A Lax pair}
\label{sec:Lax}

In the smooth setting the curvature lines of a surface of constant negative Gau{\ss} curvature are the sum and difference of the arclength parametrized asymptotic lines. To discover a Lax pair for cK-nets it is therefore natural to consider the nets formed by the diagonals of a K-net formed by skew rhombi ($\delta_u = - \delta_v$).

Given the Lax pair \eqref{eq:KLaxPair} for K-nets consider the following matrix
\begin{equation}
\mathcal L = V_1U = \quadmatrix{\cot\frac{\delta}2\frac{H_1}H + \tan\frac{\delta}2 H_1H_{12}}{i(\lambda -\frac{HH_{12}}\lambda)}{i(\lambda - \frac1{\lambda HH_{12}})}{\cot\frac{\delta}2\frac{H}{H_1} + \tan\frac{\delta}2 \frac{1}{H_1H_{12}}}
\end{equation}
with the understanding that $\delta = \delta_u = -\delta_v$.

Setting aside for a moment the fact that $\mathcal L$ arises as the product of two matrices along edges of a K-net, we can assign $\mathcal L$ to edges of a lattice and ask when this closes. 
After relabeling the entries we can set
\begin{equation}
\label{eq:LaxPair}
\begin{split}
L &=  \quadmatrix{\cot\frac{\delta_1}2\frac{l}s + \tan\frac{\delta_1}2 ls_1}{i(\lambda -\frac{ss_{1}}\lambda)}{i(\lambda - \frac1{\lambda ss_{1}})}{\cot\frac{\delta_1}2\frac{s}{l} + \tan\frac{\delta_1}2 \frac{1}{ls_{1}}}\\
M &=  \quadmatrix{\cot\frac{\delta_2}2\frac{m}s + \tan\frac{\delta_2}2 ms_2}{i(\lambda -\frac{ss_{2}}\lambda)}{i(\lambda - \frac1{\lambda ss_{2}})}{\cot\frac{\delta_2}2\frac{s}{m} + \tan\frac{\delta_2}2 \frac{1}{ms_{2}}}
\end{split}
\end{equation}
with unitary variables $s$ at vertices of a square lattice and complex functions $l$ and $m$ on edges in the first and second lattice directions.\footnote{
The matrices $L$ and $M$ can be gauged to only have edge variables. For 
\begin{equation*}
G=\quadmatrix{\sqrt{s}}{0}{0}{\sqrt{s}^{-1}}
\end{equation*}
we find
\begin{equation*}
G_1^{-1}LG = \quadmatrix{l\frac{\cot \frac{\delta_1}2}{\sqrt{s_1s}} +l \tan\frac{\delta_1}2 \sqrt{s_1s}}{i\frac{\lambda}{\sqrt{s_1s}} -i \frac{\sqrt{s_1s}}{\lambda}}{i\lambda \sqrt{s_1 s} - i\frac1{\lambda \sqrt{s_1 s}}}{\frac1l cot\frac{\delta_1}2 \sqrt{s_1 s} + \frac1l \frac{\tan\frac{\delta_1}2}{\sqrt{s_1s}}}
\end{equation*}
and a similar expression for $G_2^{-1}MG$. For $i=1,2$ choose $\sqrt{s_i s}$ as new variables on the edges.
}
The length of the edge variable $l$ (or $m$) depends on the length of the K-net edges corresponding to the $L$ (or $M$) Lax matrix. The length of a corresponding K-net edge is given by $e_i = \sin\delta_i$. If $e_i > 1$ then $\delta_i=\arcsin e_i$ is complex and $\tan\frac{\delta_i}2$ goes from real to unitary. To ensure that $L$ (or $M$) is quaternionic the length of $l$ (or $m$) is
\begin{equation}
\label{eq:ellLength}
\sqrt{\frac{\cos\left(\rho_i-\arg\tan\frac{\delta_i}2\right)+\cos\left(\rho-\arg\tan\frac{\delta_i}2\right)}{\cos\left(\rho_i+\arg\tan\frac{\delta_i}2\right)+\cos\left(\rho+\arg\tan\frac{\delta_i}2\right)}},
\end{equation}
where $s = e^{i \rho}$. When $\tan\frac{\delta_i}2$ is real this length is one as expected.

\begin{remark}
For complex $\delta_i$ these $L$ (or $M$) Lax matrices still factor into a product of $U,V$ matrices of the K-net form \eqref{eq:KLaxPair}, but each factor is no longer a quaternion, but a biquaternion. Recall that the biquaternions are given as the \emph{complex} vector space over the Pauli matrices \eqref{eq:PauliMatrices}. We will refer to both quaternionic and biquaternionic matrices of this form as K-net matrices.
\end{remark}

The compatibility condition
\begin{equation}
  \label{eq:integrabilityCondition}
 M_1L =L_2M
\end{equation}
implies $\det M_1\det L = \det L_2 \det M$. One finds $\det L =
\lambda^2 +\frac1{\lambda^2} +
\tan^2\frac{\delta_1}2+\cot^2\frac{\delta_1}2$ and $\det M = \lambda^2
+\frac1{\lambda^2} + \tan^2\frac{\delta_2}2+\cot^2\frac{\delta_2}2$. In the spirit of the
K-net case we assume that $\delta_1$ is constant in the second lattice direction and $\delta_2$ is constant in the first one. This allows \eqref{eq:integrabilityCondition} to be solved:
Given $s,s_1,s_2,l,m$, and setting $t_i = \tan\frac{\delta_i}{2}$ for notational simplicity, one finds after a long computation that
\begin{equation}
\label{eq:LaxEvolution}
\begin{split}
l_2 &= \frac{-m t_2 (s s_2 + t_1^2) + l t_1 (s s_2 + t_2^2)}{m(-lt_2(1+s s_2 t_1^2) + m(t_1 + s s_2 t_1 t_2^2))}\\
m_1 &= \frac{m t_2 (s s_1 + t_1^2) - l t_1 (s s_1 + t_2^2)}{l(lt_2(1+s s_1 t_1^2) - m(t_1 + s s_1 t_1 t_2^2))}\\
s_{12} &= s\frac{S^+ - S^-_1}{S^+-S^-_2}, \quad \mathrm{where} \\
\quad&S^+ = l^2 t_1 t_2 (1 + s s_1 t_1^2) (s s_2 + t_2^2) + m^2 t_1 t_2 (s s_1 + t_1^2)(1 + s s_2 t_2^2),\\
\quad&S^-_1 = l m (s s_1 t_1^2 + t_2^2 (2 t_1^2 + s s_2 (1 + 2 s s_1 t_1^2 + t_1^4)) + s s_1 t_1^2 t_2^4),\\
\quad&S^-_2 = l m (s s_2 t_1^2 + t_2^2 (2 t_1^2 + s s_1 (1 + 2 s s_2 t_1^2 + t_1^4)) + s s_2 t_1^2 t_2^4).
\end{split}
\end{equation}
For K-nets the zero curvature condition holds for all $\lambda$ and is not only 3D consistent, but multidimensionally consistent. The $L,M$ Lax matrix holonomy condition corresponds to the consistency of the 4D K-net system on a 4D cube. An 8-loop of K-net edges on such a cube has $\lambda$ independent holonomy and cK-net edges are given by diagonals on 2D faces. Note that in general the 4D solution cannot be extended from a 3D system as shown in the special case of a cK-net quad in Figure~\ref{fig:eightLoop}.

\begin{figure}[t]
  \centering
  \includegraphics[width=.6\hsize]{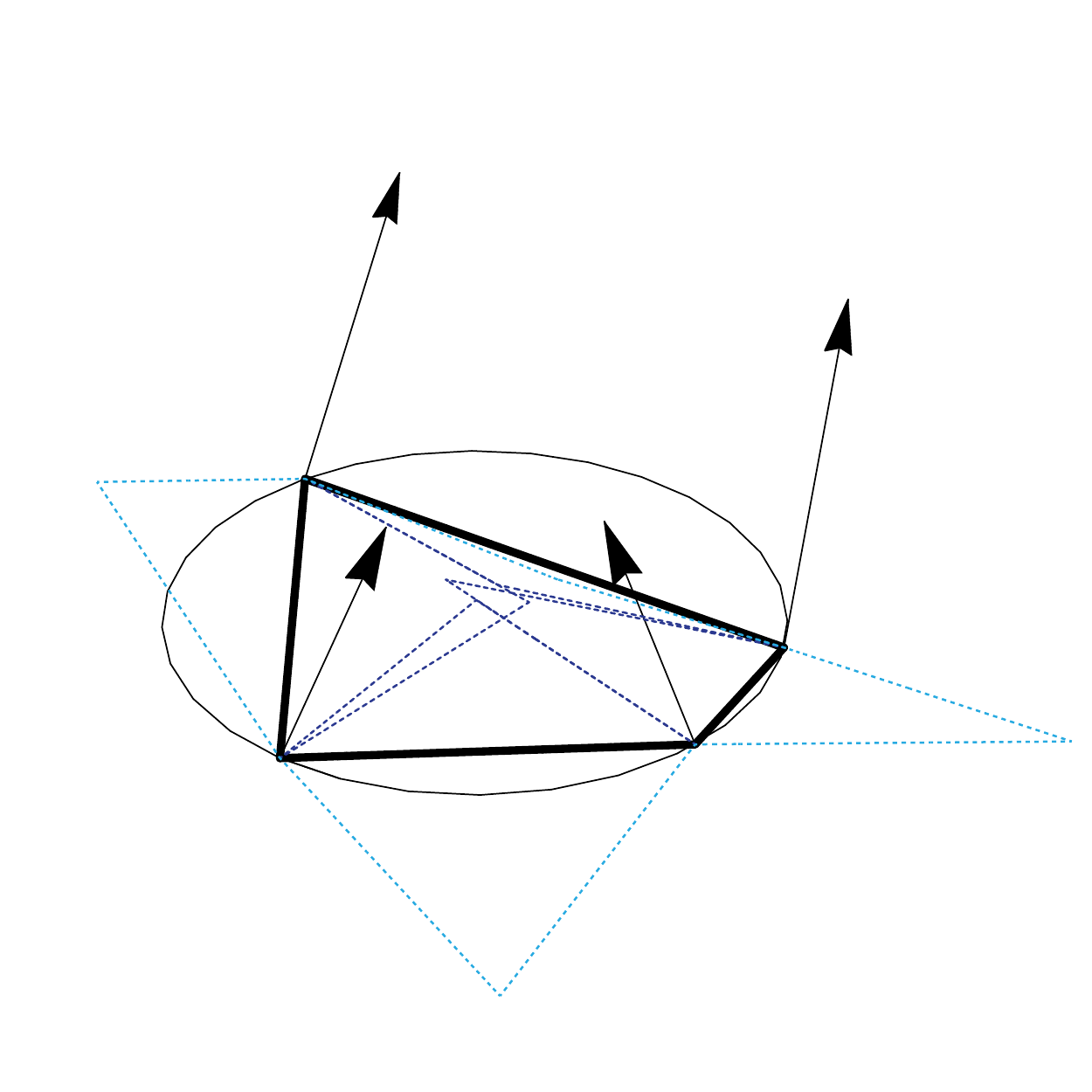}
  \caption{A quad arising from the $L,M$ Lax matrices (this particular example is a cK-net quad, see Theorem \ref{thm:ckFromLax}). Observe that each edge (thick black) is the diagonal of a folded parallelogram even though these four parallelograms do not meet in a vertex. The zero curvature condition for the $L,M$ Lax matrices follows from the zero curvature condition for K-nets by looking at one of the K-net 8-edge loops (dashed dark and light blue).}
  \label{fig:eightLoop}
\end{figure}

\begin{theorem}
Let $\Phi:\Z^2\to\R^3$ be a solution to $\Phi_1 = L\Phi$, $\Phi_2 = M\Phi$ with $L$ and $M$ as in \eqref{eq:LaxPair} solving the integrability condition \eqref{eq:integrabilityCondition}. Then $f,n:\Z^2\to \R^3$ given by
\begin{equation}
\label{eq:Sym}
  f = 2\left[\Phi^{-1}\frac{\partial}{\partial t}\Phi\right]^{\mathrm{tr}=0}, \quad n = -i \Phi^{-1}\sigma_3 \Phi, \quad \lambda = e^t
\end{equation}
is an edge-constraint net for each $t \in \R$.
\end{theorem}
\begin{proof}
We have seen that $f$ is well defined. In general, that the edge-constraint is satisfied can be checked algebraically. For real valued $\delta_i$ there is a geometric argument: the edges of these nets arise as diagonals of folded parallelograms with vertex planes perpendicular to $n$ at incident vertices. Since folded parallelograms have $180^\circ$ rotational symmetry, the edge-constraint is satisfied.
\end{proof}
Since we can choose the spectral parameter $\lambda=e^t$ freely, the above Lax pair gives rise to a one parameter \emph{associated family} of nets.

\begin{figure}[t]
  \centering
  \includegraphics[width=.5\hsize]{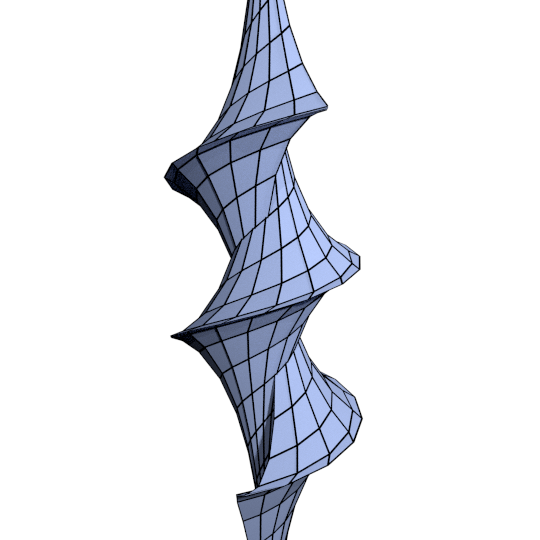}\includegraphics[width=.5\hsize]{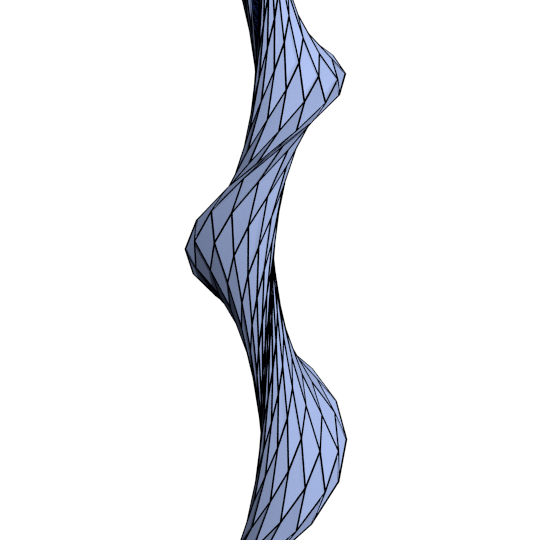}%
  \caption{Two members of the associated family of the pseudosphere. A closed form expression for this family is given by \eqref{eq:pseudoAssocFam}.}
  \label{fig:Pseudo}
\end{figure}

\subsection{cK-nets and their associated families from the Lax representation}
\label{sec:cKGeometry}

Let us investigate some geometric properties of these nets.
In particular, we will now see that the edge-constraint net $f$ with $n$ has
Gau{\ss} curvature $K=-1$ and is circular for $\lambda = 1$, thus giving us a way to generate cK-nets.

\begin{theorem}
\label{thm:ckFromLax}
  The net $f:\Z^2\to\R^3$ that arises from the system
  \eqref{eq:LaxPair} and \eqref{eq:integrabilityCondition} has Gau\ss{}
  curvature $K=-1$ for all values of the spectral parameter $\lambda =
  e^t$. It is circular if $\lambda =1$.
\end{theorem}
\begin{proof}
  The proof is a direct calculation. Solve the system for one quad and
  look at the curvature as well as the quaternionic
  cross-ratio:
  The net $f$ and its Gau\ss{} map are given by the Sym formula
  \eqref{eq:Sym}. The curvature can then
  be computed by \eqref{eq:GaussCurvature}.  To show the the
  circularity one can utilize the fact that the cross-ratio $cr =
  \frac{a-b}{b-c}\frac{c-d}{d-a}$ of four complex numbers $a,b,c$, and
  $d$ is real if and only if the points are concircular. For four
  points in $\R^3$ given as imaginary quaternions $A,B,C$, and $D$ this
  translates into $(A-B)(B-C)^{-1}(C-D)(D-A)^{-1}\in \R$ -- see, e.g., \cite{Hoffmann:1999vb}. This quantity can be computed to be real for $\lambda=1$.
\end{proof}
\begin{remark}
While in the K-net case all members of the associated family are A-nets (in discrete asymptotic line parametrization) here we have C-nets (discrete curvature line parametrization) when $\lambda=1$. Other values of $\lambda$ still give rise to edge-constraint nets of constant negative Gau\ss{} curvature (as shown in Theorem \ref{thm:ckFromLax}), but in general the quadrilaterals are no longer planar. This is expected since, in the smooth setting, in asymptotic parametrization the associated family maps $\Vert f_x\Vert \to \lambda \Vert f_x\Vert$ and $\Vert f_y \Vert \to \frac1\lambda\Vert f_y\Vert$ and only for strict Chebyshev parametrization ( $\Vert f_x\Vert = \Vert f_y \Vert $ ) do the sum and difference of the asymptotic directions give rise to curvature directions everywhere.
\end{remark}
In fact, every cK-net quad arises from the Lax system \eqref{eq:LaxPair} and \eqref{eq:integrabilityCondition}.

\begin{theorem}
\label{thm:ckIsLax}
Given a circular edge-constraint quadrilateral $f$ with parallel Gau{\ss} map $n$ and Gau{\ss} curvature $K = -1$, then $f$ arises from a Lax representation as in \eqref{eq:LaxPair}.
\end{theorem}
\begin{proof}

The proof follows in two steps. First we show that cK-net quads are described by Cauchy data, it is uniquely determined by three vertices and one normal. Then, we show how to explicitly determine the parameters of the Lax matrices \eqref{eq:LaxPair} from two meeting edges and a normal, i.e., the same data. Therefore, the given cK-net quad and the one arising from the Lax pair evolution equations \eqref{eq:LaxEvolution} must coincide.

$\left. 1\right)$ Given a circular quad with vertices circumfencing radius $r$, $f
  =r e^{i\phi}, f_1 =r e^{i\phi_1}$, and $f_{12} = re^{i\phi_{12}}$
  and an initial normal $n =
  (\sin\alpha\cos\beta,\sin\alpha\sin\beta,\cos\alpha)$ one can
  calculate parallel normals $n_1,n_2,n_{12}$ by the condition that
  $n_1 -n \parallel f_1-f$ and likewise on the other edges. The
  Gau{\ss} curvature can then be found to be
  $$ K = \frac{\sin^2(\alpha) \csc \left(\frac{1}{2} (\phi-{\phi_1}+{\phi_{12}}-{\phi_2})\right) \sin
    \left(\frac{1}{2} (4 \beta-3 \phi-{\phi_1}+{\phi_{12}}-{\phi_2})\right)}{r^2}.$$
  Solving for $\phi_{12}$ one finds
  $$\phi_{12} = \textstyle
  2 \arctan\left(\frac{-r^2 \csc ^2(\alpha) \sin \left(\frac{1}{2}
        (\phi-{\phi_1}-{\phi_2})\right)-\sin \left(\frac{1}{2} (4 \beta-3
        \phi-{\phi_1}-{\phi_2})\right)}{r^2 \csc ^2(\alpha) \cos \left(\frac{1}{2}
        (\phi-{\phi_1}-{\phi_2})\right)+\cos \left(\frac{1}{2} (4 \beta-3
        \phi-{\phi_1}-{\phi_2})\right)}\right)
  $$
  or
  $$
  f_{12} = r e^{-i (\phi-{\phi_1}-{\phi_2})} \frac{r^2 +\sin^2(\alpha)e^{2 i(
      \phi-\beta)}}{r^2 + \sin^2(\alpha)e^{-2 i (\phi-\beta)}}.
  $$
  So given three initial points and a normal at the middle one there are
  a unique fourth vertex and unique normals at the remaining points that
  furnish a circular net with parallel normals that has Gau{\ss} curvature minus one.

$\left.2\right)$ Let $f_i-f$ be an edge with length $d$, dihedral angle $\delta_i$ (possibly
  non-real) that corresponds to the associated virtual edges, and
  $\phi$ be the angle the edge makes with its incident normals, then
  \begin{equation}
  \sin^2\delta_i = \frac{d^2}{4} + \cos^2\phi.
  \end{equation}
  We also find that
    \begin{equation}
    \begin{split}
        d^2 &= 4 \cos^2(\frac{\rho + \rho_i}{2}) \sin^2\delta_i \quad \mathrm{and}  \\
         \cos(\phi) &= - \sin(\frac{\rho + \rho_i}{2}) \sin \delta_i,
    \end{split}
    \end{equation}
  where $s = e^{i \rho}$ and each $s_i = e^{i \rho_i}$.
  These yield:
  \begin{equation}
  \begin{split}
  \delta_i &= \arcsin\sqrt{\frac{d^2}{4} + \cos^2\phi} \\
  \rho_i &= -\rho + 2\arcsin(-\cos(\phi) \csc(\delta_i)).
  \end{split}
  \end{equation}

  This gives a way to calculate both $\delta_i$ and $\rho_i$ (and thus
  $s_i$) from the edge length $d$ and the angle with its normals
  $\phi$, after an initial choice of $s = e^{i\rho}$.

The final degree of freedom in the Lax matrix is the argument of the edge variable $l$ (or $m$), which encodes the rotation of the edge about one of its incident normals. This can be chosen to align the edge arising from the Sym-formula Lax matrix with the specified edge. The above formulas for the entries ensure that the radicant in
  the formula \eqref{eq:ellLength} for the length of the edge variable $l$ is
  nonnegative. We find the following equivalent inequalities:
  \begin{equation}
	\begin{split}
	0 &\leq \frac{\cos(\frac{\rho + \rho_i}{2} - \arg\tan\frac{\delta_i}2)}{\cos(\frac{\rho+\rho_i}{2} + \arg\tan\frac{\delta_i}2)}\\
	\iff 0 &\leq d^2 \cos^2(\arg\tan\frac{\delta_i}2) - \cos^2\phi \sin^2(\arg\tan\frac{\delta_i}2)\\
	\iff 0 &\leq 4 \sin^2\phi.
	\end{split}
  \end{equation}
The last inequality (found using that $\sec^2(\arg\tan\frac{\delta_i}2) = \sin^2\delta_i$) is clearly satisfied, so the radicant is nonnegative.
Therefore, we get Lax matrices of the correct form for each edge.

\end{proof}

\subsection{B\"acklund transformations of cK-nets}
\label{sec:BT}
A smooth B\"acklund transformation is very geometric and characterized by the conditions that (i) corresponding points lie in their respective tangent planes, (ii) are in constant distance, and (iii) that corresponding normals form a constant angle. Furthermore, asymptotic lines and curvature lines are preserved.

The discrete B\"acklund transformation for discrete K-nets in asymptotic para\-metri\-zation (Definition \ref{def:knetBT}) is also characterized by these conditions and preserves the discrete asymptotic parametrization (note that there is some condition on the data, like the distance must equal the sine of the angle the normals make).

In this section we introduce a B\"acklund transformation for discrete K-nets in curvature line parametrization (cK-nets), which is characterized by the same geometric conditions and preserves the discrete curvature line parametrization. The B\"acklund transformation also carries over to the associated family of the cK-nets, generating edge-constraint nets of constant negative Gau{\ss} curvature in more general types of parametrizations.

Algebraically, a single B\"acklund transformation $(\tilde f, \tilde n)$ of a cK-net $(f,n)$ is determined by multiplying its frame $\Phi$ by one of the $U,V$ K-net matrices depending on a B\"acklund parameter $\alpha$ and the same spectral parameter $\lambda$, e.g., $\tilde \Phi(k,\ell,\lambda) = U(\alpha, \lambda) \Phi(k,\ell,\lambda)$. The resulting frame can then be integrated via the Sym-Bobenko formula \eqref{eq:Sym}. As each cK-net frame factors into a sequence of K-net matrices, existence of B\"acklund transformations and a Bianchi permutability theorem follow from the corresponding theorems for K-nets given at the end of Section \ref{sec:K-nets}.

\begin{figure}[t]
  \centering
  \includegraphics[width=\hsize]{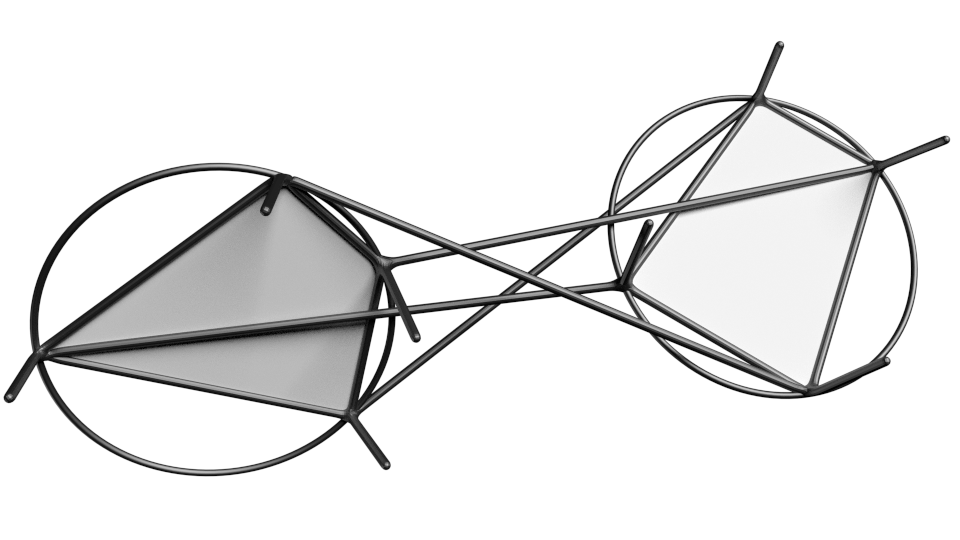}
  \caption{A cK-net quad with its B\"acklund transform.}
  \label{fig:cube}
\end{figure}

\begin{theorem} Let $f:\Z^t \to \R^3$ with Gau{\ss} map $n$ and Gau{\ss} curvature $K=-1$ be a cK-net. Then (up to a global degree of freedom fixing an initial normal)
\begin{enumerate}
\item For every angle $\alpha \neq 0 \in (-\pi, \pi)$ there exists
  a unique cK-net $\tilde f$ with
  Gau\ss{} map $\tilde n$ such that $\Vert \tilde f - f\Vert =
  \sin\alpha$, $\angle(n,\tilde n) = \alpha$, and $(\tilde f - f)\perp n,\tilde n$. The nets $\tilde f$ are called the \emph{B\"acklund transforms} of $f$.
\item  For every pair of B\"acklund transforms $\hat f$ and $\tilde f$ with
  parameters $\hat \alpha$ and $\tilde \alpha$, respectively, there exists a
  unique cK-net $\hat{\tilde f}$ that is a
  $\tilde\alpha$-B\"acklund transform of $\hat f$ as well as a
  $\hat\alpha$-B\"acklund transform of $\tilde f$.
\end{enumerate}
\begin{proof}
The result follows from the corresponding theorem for asymptotic K-nets by factoring the cK-net frame into a sequence of K-net matrices.

For completeness we now describe the evolution equations for the Lax matrix variables $s, l, m$ to $\tilde s, \tilde l, \tilde m$.
For B\"acklund matrices $U$, solving $U_1 L = \tilde L U$ and $U_2 M = \tilde M U$ yields
\begin{equation}
\label{eq:singleBTEvolution}
\begin{split}
\tilde s_1 &= \frac{\tilde s \sin (\alpha ) \left(s s_1+t_1^2\right)-l t_1 \left(\left(s s_1-1\right) \cos (\alpha )+s s_1+1\right)}{l
   \left(l \sin (\alpha ) \left(s s_1 t_1^2+1\right)+\tilde s t_1 \left(\left(s s_1-1\right) \cos (\alpha )-s s_1-1\right)\right)}\\
\tilde l &= \frac{s \left(\tilde s-l t_1 \cot \left(\frac{\alpha }{2}\right)\right)}{l-\tilde s t_1 \cot \left(\frac{\alpha }{2}\right)}\\
\tilde s_2 &= \frac{\tilde s \sin (\alpha ) \left(s s_2+t_2^2\right)-m t_2 \left(\left(s s_2-1\right) \cos (\alpha )+s s_2+1\right)}{m \left(m \sin (\alpha )
   \left(s s_2 t_2^2+1\right)+\tilde s t_2 \left(\left(s s_2-1\right) \cos (\alpha )-s s_2-1\right)\right)}\\
\tilde m &= \frac{s \left(\tilde s-m t_2 \cot \left(\frac{\alpha }{2}\right)\right)}{m-\tilde s t_2 \cot \left(\frac{\alpha }{2}\right)}
\end{split}
\end{equation}
That $\tilde s_{12} = \tilde s_{21}$ follows from $s_{12} = s_{21}$. Therefore, up to an initial choice of $\tilde s = e^{i \theta}$ at one point fixing an initial normal vector, the evolution is uniquely determined.

The immersion and Gau{\ss} map of the transform are given by
\begin{equation}
\label{eq:singleBTImmGauss}
\begin{split}
\tilde f(k,\ell) &= f(k,\ell) + \sin(\alpha) \qIm{	\Phi^{-1}(k,\ell) \quadmatrix{0}{i \frac{s(k,\ell)}{\tilde s(k,\ell)}}{i \frac{\tilde s(k,\ell)}{s(k,\ell)}}{0} \Phi(k,\ell)} \quad \mathrm{and} \\
\tilde n(k,\ell) &= - i \Phi^{-1}(k,\ell) U^{-1}(k,\ell) \sigma_3 U (k,\ell) \Phi(k,\ell),
\end{split}
\end{equation}
so from an initial cK-net we only require the $\tilde s$ variables to describe its B\"acklund transform.
\end{proof}
\end{theorem}

\begin{remark}
The B\"acklund transformation works for nets $(f^\lambda, n^\lambda)$ in the associated family with spectral parameter $\lambda = e^t$, giving rise to more generally parametrized edge-constraint nets of constant negative Gau{\ss} curvature. The equations are given by \eqref{eq:singleBTImmGauss}, with $(f,n)$ replaced by $(f^\lambda, n^\lambda)$, the frame $\Phi$ replaced with $\Phi^\lambda$, and the constant distance $\sin\alpha = \| \tilde f - f \|$ replaced by $\frac{\sin\alpha}{\cosh t - \cos\alpha \sinh t} = \|\tilde f^\lambda - f^\lambda\|$. Also, the constant angle between normals is given by $\arccos\frac{\cos \alpha \cosh t - \sinh t}{\cosh t - \cos\alpha \sinh t}$, so the relationship $\| \tilde f^\lambda - f^\lambda \|^2 + (\tilde n \cdot n)^2 = 1$ still holds.

As in the smooth setting, the B\"acklund parameter $\alpha$ and spectral parameter $\lambda$ are related; setting $\lambda = 1$ and varying $\alpha$ generates a family of cK-net surfaces, conversely, fixing $\alpha$ and varying $\lambda$ generates a similar family of surfaces in more general parametrizations. For an explicit example see the remarks after Theorem \ref{thm:singleBT}.
\end{remark}

\begin{example}
  The construction of the Pseudosphere given at the start of Setion \ref{sec:cK-nets} is in fact a B\"acklund
  transformation of the straight line (details can be found in Section \ref{sec:lineSingleBT}).  Figure~\ref{fig:Kuen} shows a
  discrete Kuen surface; it arises as a B\"acklund transformation of
  the Pseudosphere. Shown in Figure~\ref{fig:clock} are the B\"acklund
  transformations of the Pseudosphere aligned by their angular
  parameter. Since the B\"acklund transformation is invertible one
  finds both the straight line and the Kuen surface therein.
\end{example}

\begin{figure}[t]
  \centering
  \includegraphics[width=.5\hsize]{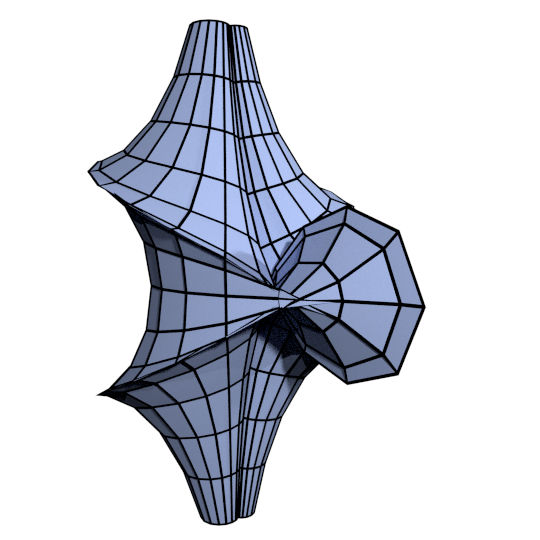}
  \caption{A discrete circular Kuen surface. The coordinate polygons in one direction are planar.}
  \label{fig:Kuen}
\end{figure}

\subsubsection{Double B\"acklund transformations and a remark on multidimensional consistency}
The 3D compatibility cube arising from a quadrilateral of a cK-net $f$ together with its B\"acklund transform $\tilde f$ is not the usual 3D consistency cube, since the equation on the side quadrilaterals is different from that on the top/bottom pair of quadrilaterals. Figure~\ref{fig:cube} shows such a cube together with normals.
  However, if $\tilde f$ is a double B\"acklund transform with parameters $\alpha$ and $-\alpha$ the resulting cube is a familiar 3D consistent cube, as it is circular on all sides. This observation immediately gives the following result.
\begin{corollary}
  The Kuen surface that arises from the Pseudosphere with parameter
  $\frac\pi2$ has planar coordinate polygons in one lattice direction.
\end{corollary}
\begin{proof}
  Since the Pseudosphere arises as a B\"acklund transform from the
  straight line with parameter $-\frac\pi2$, the Kuen surface can be
  viewed as a double B\"acklund transform of the line that is formed
  by cubes with circular sides. Thus the line and all parameter
  polygons of the Kuen net in one direction are sides of a strip of
  circular quadrilaterals, which clearly must be planar (since it
  contains the straight line in its border).
\end{proof}

\begin{figure}[t]
  \centering
  \includegraphics[width=.49\hsize]{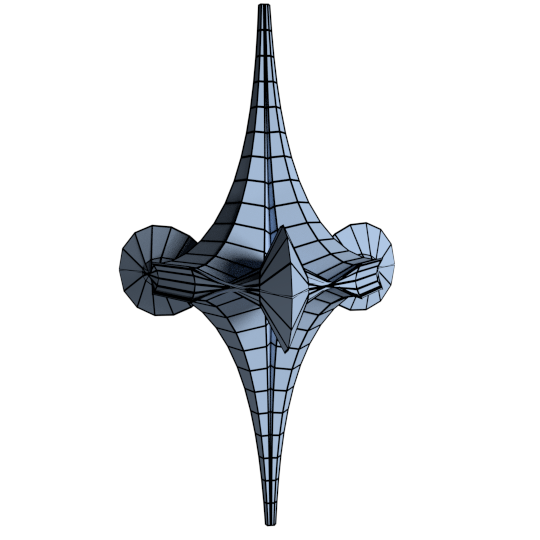}
  \includegraphics[width=.49\hsize]{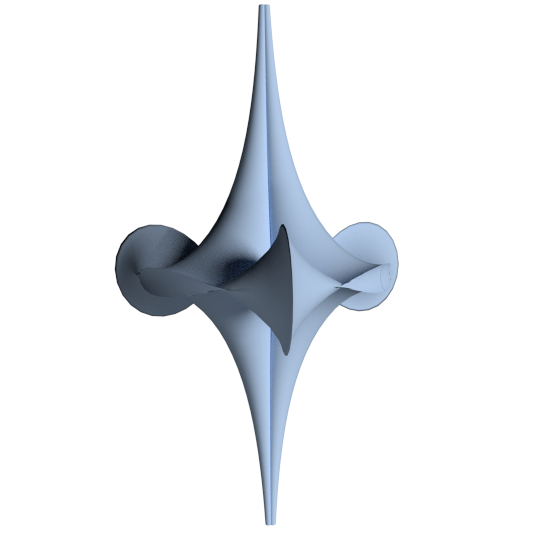}
  \caption{A nonfactorizable double B\"acklund transform of the vacuum sided (left) with the corresponding analytic solution (right). The coordinate polygons in one direction are planar.}
  \label{fig:complexBT}
\end{figure}

Double B\"acklund transformations with real parameters $\alpha$ and
$-\alpha$ as above can clearly be represented by multiplying the frame by a Lax matrix of the
type of $L$ or $M$, factorizable into a product of quaternionic K-net matrices. However, recall (see the discussion around \eqref{eq:ellLength}) that the matrices $L,M$ are more general and allow for $\alpha$ to be complex valued. Geometrically one can think of this as follows: For a single B\"acklund transformation the distance
of the transformed points to their preimages in $\R^3$ must be $\sin\alpha$. Thus, if the distance is larger than one, the angle
$\alpha$ is no longer real and so neither is the transformed surface. However
with a second B\"acklund transformation one can achieve a real
solution again.  These double B\"acklund transformations that do not
factor into two real ones are in fact well known in the continuous
case. Figure~\ref{fig:complexBT} shows an example, a so-called \emph{breather} solution (the name stems from the behavior of the corresponding solution to the sine-Gordon equation).

Schief has described these B\"acklund transformations in \cite{Schief:2003ug} in the realm of circular nets but was not able to get the single step transformations since they do not give rise to circular 3D compatibility cubes.

\subsection{Explicit discrete parametrizations for transformations of the straight line}
We now present closed form equations for some B\"acklund transformations of the straight line, together with their associated families. As in the smooth setting, single transformations give rise to Beltrami's pseudosphere (Figure \ref{fig:Pseudo}) and Dini's surfaces (Figure \ref{fig:DiniCompare}), while double transformations (with opposite parameters $\pm \alpha$) give rise to breather surfaces (Figure \ref{fig:complexBT}) and Kuen's surface (Figure \ref{fig:Kuen}).

\subsubsection{Discrete Dini's surfaces}
\label{sec:lineSingleBT}
The straight line (notated with a zero subscript) can be represented by edge and vertex functions $l_0(k,\ell) = m_0(k,\ell) = s_0(k,\ell) = (-1)^\ell$ plugged into the Lax matrices \eqref{eq:LaxPair} and integrated via the Sym-Bobenko formula \eqref{eq:Sym}. For simplicity throughout this section we assume that the parameter line functions $\delta_1(k), \delta_2(\ell)$ are constants. Explicitly, the corresponding immersion and Gau{\ss} map $(f_0, n_0)$ with spectral parameter $\lambda = e^t$ are given by
\begin{equation}
\label{eq:straightLine}
\begin{split}
&f_0(k,\ell,t) = \begin{bmatrix}-2 x(k,\ell,t)\\ 0\\ 0 \end{bmatrix} \quad \mathrm{with} \quad 
n_0(k,\ell,t) = \begin{bmatrix}0\\ \Im \omega(k,\ell,t)\\ \Re\omega(k,\ell,t) \end{bmatrix}, \mathrm{where} \\
&x(k,\ell,t) = k \left(\frac{\cosh t \sin \delta
   _1}{1+\sinh ^2 t \sin ^2\delta _1} \right)
   +
  \ell \left(
   \frac{\sinh t \sin \delta _2 \cos \delta
   _2}{1+\sinh ^2 t \sin ^2\delta _2} \right) \quad \mathrm{and}\\
&\omega(k,\ell,t) = \left(\frac{i + \sinh t \sin \delta
   _1}{i -\sinh t \sin \delta _1}\right)^k
  \left( \frac{i + \cosh t \tan \delta _2}{i -\cosh t \tan
   \delta _2}\right)^\ell.
\end{split}
\end{equation}
When $t = 0$ we recover a degenerate cK-net of a familiar form
\begin{equation}
f_0(k,\ell) = \begin{bmatrix}-2 k \sin \delta_1\\0\\0\end{bmatrix} \quad \mathrm{with} \quad 
n_0(k,\ell) = \begin{bmatrix}0\\ - \sin(2 \ell \delta_2)\\ \cos (2 \ell \delta_2)\end{bmatrix}.
\end{equation}

The single B\"acklund transformation of the straight line can be solved in full generality in closed form; the derivation was performed using a mixture of hand and symbolic computation. Throughout this section we denote by subscript b the single B\"acklund transform of the straight line. For an initial choice $s_b(0,0) = e^{i \theta}, \theta \neq 0 \in (-\pi,\pi)$, the evolution recursion formulas \eqref{eq:singleBTEvolution} can be solved yielding
\begin{equation}
\label{eq:singleBTRecursion}
\begin{split}
s_b(k,\ell) &= (-1)^\ell \left( 	-1 + \frac{2}{1 - i e^{\chi(k,\ell)}}	\right), \quad \mathrm{where} \\
\chi(k,\ell,\alpha) &= 
\log \tan \frac{\theta}{2} +
k \log \frac{\sin \alpha +\sin \delta _1}{\sin \alpha -\sin \delta _1} -
\ell \log \frac{\sin \left(\alpha -\delta _2\right)}{\sin \left(\alpha +\delta _2\right)}.
\end{split}
\end{equation}
Solving for the immersion and Gau{\ss} map explicitly using \eqref{eq:singleBTImmGauss} we arrive at the following theorem.
\begin{theorem}
\label{thm:singleBT}
The B\"acklund transformation of the straight line ($f_0$ and $\omega$ as in \eqref{eq:straightLine}) with parameter $\alpha \in (-\pi,\pi)$, together with its associated family with spectral parameter $\lambda = e^t \in \R$, is given by
\begin{equation}
\begin{split}
f_b(k,\ell,\alpha,t) &= f_0(k,\ell,t) + \frac{\sin\alpha}{\cosh t - \cos \alpha \sinh t} \begin{bmatrix}\tanh \chi \\ -\sech\chi \Re \omega \\ \sech\chi \Im \omega \end{bmatrix} \mathrm{with} \\
n_b(k,\ell,\alpha,t) &= \frac{\sin\alpha}{\cosh t - \cos \alpha \sinh t} \begin{bmatrix}\sech \chi \\ \Im\hat\omega \\ \Re\hat\omega \end{bmatrix},
\end{split}
\end{equation}
with $\chi$ defined as in \eqref{eq:singleBTRecursion} and for simplicity we set
\begin{equation}
\begin{split}
\hat \omega &= \omega (i \tanh \chi + (\cosh t \cot \alpha - \csc \alpha \sinh t)).
\end{split}
\end{equation}
\end{theorem}

\begin{remark}
In the most general case where the parameter line functions $\delta_1(k)$ and $\delta_2(\ell)$ vary, we have (up to reversing summation/product indices based on the signs of $k,\ell$):
\begin{equation}
\begin{split}
\chi(k,\ell,\alpha) &=
\log \tan \frac{\theta}{2} +
\sum _{s=0}^{k-1} \log \frac{\sin \alpha +\sin \delta _1(s)}{\sin \alpha -\sin \delta _1(s)} -
\sum _{s=0}^{\ell-1} \log \frac{\sin \left(\alpha -\delta _2(s)\right)}{\sin \left(\alpha +\delta _2(s)\right)}, \\
\omega(k,\ell,t) &= \prod _{s=0}^{k-1} \frac{i + \sinh t \sin \delta
   _1(s)}{i -\sinh t \sin \delta _1(s)}
   \prod
   _{s=0}^{\ell-1} \frac{i + \cosh t \tan \delta _2(s)}{i -\cosh t \tan
   \delta _2(s)}, \quad \mathrm{and}\\
x(k,\ell,t) &=\sum _{s=0}^{k-1} \frac{\cosh t \sin \delta
   _1(s)}{1+\sinh ^2 t \sin ^2\delta _1(s)}
   +
   \sum _{s=0}^{\ell-1}
   \frac{\sinh t \sin \delta _2(s) \cos \delta
   _2(s)}{1+\sinh ^2 t \sin ^2\delta _2(s)}.
\end{split}
\end{equation}
\end{remark}

We wish to highlight two special cases of the above theorem; the first provide a discrete analogue of Dini's surfaces in curvature line coordinates given by a family of B\"acklund transformations, while the latter provides a discrete analogue of Dini's surfaces in more general coordinates given by an associated family (see Figure~\ref{fig:DiniCompare}). 
\begin{figure}[t]
  \centering
  \includegraphics[width=.5\hsize]{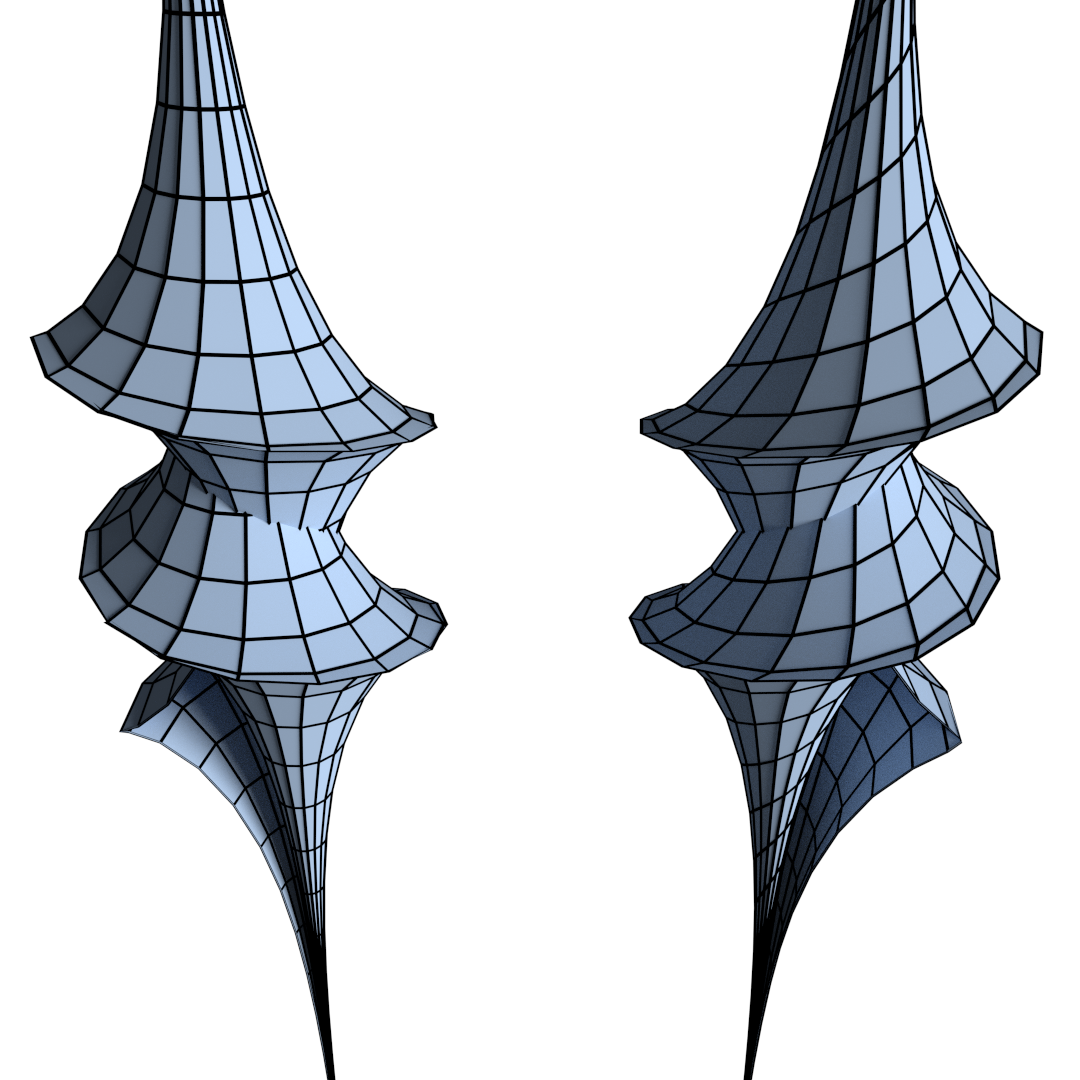}
  \caption{Two Dini nets. Note the subtle difference in how the cusp line aligns with the parameter polygons in the right version but not in the left.}
  \label{fig:DiniCompare}
\end{figure}
\begin{corollary}
Setting $t = 0$ in Theorem \ref{thm:singleBT} yields the B\"acklund transformations of the straight line that are all cK-nets. 
\begin{equation}
\begin{split}
f^{cK}_\text{dini}(k,\ell) &= \begin{bmatrix}-2 k \sin\delta_1\\0\\0\end{bmatrix} \, + \sin\alpha \begin{bmatrix} \tanh\chi \\ -\sech\chi \cos(2 \ell \delta_2) \\ - \sech \chi \sin(2 \ell \delta_2) \end{bmatrix} \quad \mathrm{with} \\
n^{cK}_\text{dini}(k,\ell)&= \begin{bmatrix} \sech \chi \sin\alpha \\
			\tanh\chi \cos(2 \ell \delta_2) \sin\alpha -\cos\alpha \sin(2 \ell \delta_2)\\
			\tanh\chi \sin(2 \ell \delta_2) \sin \alpha + \cos \alpha \cos(2 \ell \delta_2) \end{bmatrix}.
\end{split}
\end{equation}
\end{corollary}

\begin{corollary}
Setting $\alpha = -\frac\pi2$ in Theorem \ref{thm:singleBT} yields the associated family of the Beltrami pseudosphere of revolution, see Figure \ref{fig:Pseudo}.
\begin{equation}
\label{eq:pseudoAssocFam}
\begin{split}
f_\text{pseudo}(k,\ell,t) &= f_0(k, \ell, t) - \sech t  \begin{bmatrix} \tanh (k \tau)\\ -\sech(k \tau) \Re \omega \\ \sech(k \tau) \Im \omega \end{bmatrix} \quad \mathrm{with} \\
n_\text{pseudo}(k,\ell,t) &= -\sech t \begin{bmatrix} \sech\chi \\ \tanh\chi \Re\omega + \sinh t \Im \omega \\ \tanh \chi \Im\omega - \sinh t\Re\omega \end{bmatrix},
\end{split}
\end{equation}
where $\tau = \log{\frac{1-\sin\delta_1}{1+\sin\delta_1}}$.
\end{corollary}

\begin{remark}
It is clear that either setting $\alpha = -\frac\pi2$ in the cK-net family of Dini's surfaces or setting $t=0$ in the associated family of the Pseudosphere we recover, after the change of variables $\epsilon = -2\sin\delta_1$ and $\phi = 2 \delta_2$, the closed form of the Pseudosphere given in Remark \ref{rem:pseudoFormulas}.
\end{remark}

\subsubsection{Special double transformations of the straight line}
In this section we give closed form expressions for double B\"acklund transformations of the straight line with opposite (real or complex) parameters $\alpha$ and $-\alpha$, together with their associated families. Throughout we assume that $\alpha = \arcsin{d}$ for $d \geq 1$, so in particular, we can define $\mu$ by $e^{i \mu} = \tan\frac\alpha 2$.

Such transformations are given by multiplying the straight line frame by a cK-net Lax matrix with unitary vertex variables $s_0, s_{db}$ and an edge variable $s_b$ that is only unitary for $e^{i \mu} = \pm 1$

\begin{equation}
B = \quadmatrix{s_b (e^{i \mu} s_{db} + \frac{1}{e^{i \mu}s_0})}{i (\lambda - \frac{s_{db}s_0}{\lambda})}{i (\lambda - \frac{1}{\lambda s_{db}s_0})}{\frac{1}{s_b}(\frac{e^{i \mu}}{s_{db}} + \frac{s_0}{e^{i \mu}})}.
\end{equation}

Given the Lax matrices for the straight line and the single B\"acklund transformation variable $s_b(k,\ell)$ (which becomes the edge quantity here), we can solve the recurrence relations \eqref{eq:LaxEvolution} governing the evolution variable $s_{db}$. For simplicity we assume that $s_b(0,0) = i$, $s_{db}(0,0) = 1$, and that $\delta_1, \delta_2$ are constant. For $e^{i \mu} \neq \pm 1$ we find

\begin{equation}
\begin{split}
s_0(k,\ell) &= (-1)^\ell,\\
s_b(k,\ell) &= (-1)^\ell \left(-1 + \frac{2}{1 - i e^{-(i \ell \kappa + k \tau)}}\right), \quad \mathrm{and} \\
s_{db}(k,\ell) &= (-1)^\ell \left( -1 + \frac{2}{1 - i (\cot\mu \sech(k \tau) \sin(\ell \kappa))}\right),
\end{split}
\end{equation}
where $\kappa = 2 \arctan(\sin\mu \tan \delta_2)$ and $\tau = \log\frac{1-\sin\delta_1 \cos\mu}{1+\sin\delta_1\cos\mu}$. Note that we performed a change of variables to split $\chi$ into a $\kappa$ and $\tau$ part.

\begin{theorem}
The immersion and Gau{\ss} map for the stationary breather with B\"acklund parameter $\mu \neq 0,\pi \in [0, 2\pi)$ and associated family parameter $\lambda = e^t$ of the straight line (with $f_0, \omega$ as in \eqref{eq:straightLine}) are given by
\begin{equation}
\begin{split}
&f_{breather}(k,\ell,\mu,t) = f_0(k,\ell,t) + \\
&\quad 2 A \begin{bmatrix}
\cos \mu \sinh t \sech(k \tau ) \sin (\ell \kappa) \cos (\ell \kappa)-\sin \mu \cosh t \sinh (k \tau )\\
\Im\omega \sin (\ell \kappa)-\Re\omega (\cos \mu \sinh t \tanh (k \tau ) \sin (\ell \kappa)+\sin \mu \cosh t \cos (\ell \kappa))\\
\Im\omega (\cos \mu \sinh t \tanh (k \tau ) \sin (\ell \kappa)+\sin \mu \cosh t \cos (\ell \kappa))+\Re\omega \sin (\ell \kappa)
\end{bmatrix}
\end{split}
\end{equation}
and
\begin{equation}
\begin{split}
&n_{breather}(k,\ell,\mu,t) = \\
&\quad \frac {A} {2 \cosh(k \tau)} \begin{bmatrix}
4 (\sin \mu \sinh t \cosh (k \tau ) \cos (\ell \kappa)+\cos \mu \cosh t \sinh (k \tau ) \sin (\ell \kappa)) \\
\Re\omega B - \Im\omega C\\
-\Re\omega C - \Im\omega B,
\end{bmatrix}
\end{split}
\end{equation}
where
\begin{equation}
\begin{split}
A &=\frac{\sin(2\mu) \cosh(k \tau)}{(\cos(2\mu) + \cosh(2 t))(\cos^2\mu \sin^2(\ell \kappa) + \cosh^2(k \tau) \sin^2\mu)}, \\
B &= 2 (\cos \mu \cosh t \sin (2 \ell \kappa)-\sin \mu \sinh t \sinh (2 k \tau )) \quad \mathrm{and} \\
C &= \sin ^2(\ell \kappa) (\sin (2 \mu )+\cot \mu (\cosh (2 t)+1)) \\ & \quad \quad -\cosh ^2(k \tau ) (\sin (2 \mu )+\tan \mu (1-\cosh (2 t))).
\end{split}
\end{equation}
\end{theorem}

\begin{remark}
When $t=0$ the resulting cK-net immersions agree with those found by Schief \cite{Schief:2003ug}. As he notes, for each rational number $0<q<1$, setting $\mu =  -\arcsin(\cot\delta_2\tan(\delta_2 q)$ generates a surface that is closed in one lattice direction. The breather cK-net with $q = \frac 3 5 $ is shown in Figure \ref{fig:complexBT}.
\end{remark}

Naively setting $\mu = 0$ in the previous theorem does not yield Kuen's surface, however taking the limit as $\mu \to 0$ does. Alternatively, one could solve the recursion formulas for real $\tan\frac\alpha2 = 1$. In the following theorem we have $\tau = \log\frac{1-\sin\delta_1}{1+\sin\delta_1}$.

\begin{theorem}
Kuen's surface and its associated family $\lambda = e^t$, as a double transformation of the straight line, are given (with $f_0, \omega$ as in \eqref{eq:straightLine}) by
\begin{equation}
\begin{split}
&f_{Kuen}(k,\ell,t) = f_0(k,\ell,t) + \\
&\frac{2 \cosh(k \tau) \sech t}{\cosh^2(k \tau) + 4 \ell^2 \tan^2\delta_2}
\begin{bmatrix}
2 \ell \tan \delta_2 \tanh t \sech(k \tau )-\sinh (k \tau )\\
2 \ell \tan \delta_2 \Im\omega \sech t-\Re\omega \left(2 \ell \tan \delta_2 \tanh t \tanh (k \tau )+1\right)\\
\Im\omega \left(2 \ell \tan \delta_2 \tanh t \tanh (k \tau )+1\right)+2 \ell \tan \delta_2 \Re\omega \sech t
\end{bmatrix}
\end{split}
\end{equation}
with Gauss map
\begin{equation}
\begin{split}
&n_{Kuen}(k,\ell,t) = f_0(k,\ell,t) + \\ 
&\quad \frac{\sech^2 t}{\cosh^2(k \tau) + 4 \ell^2 \tan^2\delta_2}
\begin{bmatrix}
2 \left(2 l \tan \delta_2 \cosh t \sinh (k \tau )+\sinh t \cosh (k \tau )\right)\\
\Im\omega D + \Re\omega E\\
\Re\omega D - \Im\omega E,
\end{bmatrix}
\end{split}
\end{equation}
where
\begin{equation}
\begin{split}
D &= \left(1-\sinh ^2t\right) \cosh ^2(k \tau )-4 l^2 \tan ^2\delta_2 \cosh ^2t \quad \mathrm{and} \\
E &= 4 l \tan \delta_2 \cosh t-\sinh t \sinh (2 k \tau ).
\end{split}
\end{equation}
\end{theorem}

\begin{remark}
When $t=0$ we recover a cK-net Kuen's surface, as shown in Figure \ref{fig:Kuen}.
\end{remark}

\bibliography{ckbib}

\begin{thebibliography}{20}
\providecommand{\natexlab}[1]{#1}
\providecommand{\url}[1]{\texttt{#1}}
\expandafter\ifx\csname urlstyle\endcsname\relax
  \providecommand{\doi}[1]{doi: #1}\else
  \providecommand{\doi}{doi: \begingroup \urlstyle{rm}\Url}\fi

\bibitem[Bobenko(1994)]{Bobenko:1994tv}
A.~I. Bobenko.
\newblock {Surfaces in terms of 2 by 2 matrices: Old and new integrable cases.}
\newblock In A.~P. Fordy and J.~C. Wood, editors, \emph{Harmonic maps and
  integrable systems}, pages 83--129. Vieweg, Braunschweig/Wiesbaden, 1994.

\bibitem[Bobenko and Pinkall(1996{\natexlab{a}})]{Bobenko:1996ug}
A.~I. Bobenko and U.~Pinkall.
\newblock {Discrete surfaces with constant negative Gaussian curvature and the
  Hirota equation}.
\newblock \emph{Journal of Differential Geometry}, 43:\penalty0 527--611,
  1996{\natexlab{a}}.

\bibitem[Bobenko and Pinkall(1996{\natexlab{b}})]{Bobenko:1996vq}
A.~I. Bobenko and U.~Pinkall.
\newblock {Discrete isothermic surfaces}.
\newblock \emph{Journal f{\"u}r die reine und angewandte Mathematik}, pages
  187--208, 1996{\natexlab{b}}.

\bibitem[Bobenko and Pinkall(1999)]{Bobenko:1999us}
A.~I. Bobenko and U.~Pinkall.
\newblock {Discretization of Surfaces and Integrable Systems}.
\newblock In A.~I. Bobenko and R.~Seiler, editors, \emph{Discrete integrable
  geometry and physics}, pages 3--58. Oxford University Press, 1999.

\bibitem[Bobenko and Suris(2008)]{Bobenko:2008tn}
A.~I. Bobenko and Y.~B. Suris.
\newblock \emph{{Discrete Differential Geometry: Integrable Structure}},
  volume~98 of \emph{Graduate Studies in Mathematics}.
\newblock American Mathematical Society, 2008.

\bibitem[Bobenko et~al.(2010)Bobenko, Pottmann, and Wallner]{Bobenko:2010eg}
A.~I. Bobenko, H.~Pottmann, and J.~Wallner.
\newblock {A curvature theory for discrete surfaces based on mesh parallelity}.
\newblock \emph{Mathematische Annalen}, 348\penalty0 (1):\penalty0 1--24, 2010.

\bibitem[Dorfmeister et~al.(2009)Dorfmeister, Ivey, and
  Sterling]{Dorfmeister:2009eh}
J.~F. Dorfmeister, T.~Ivey, and I.~Sterling.
\newblock {Symmetric Pseudospherical Surfaces I: General Theory}.
\newblock \emph{Results in Mathematics}, 56\penalty0 (1-4):\penalty0 3--21,
  2009.

\bibitem[Hertrich-Jeromin et~al.(1999)Hertrich-Jeromin, Hoffmann, and
  Pinkall]{Hoffmann:1999vb}
U.~Hertrich-Jeromin, T.~Hoffmann, and U.~Pinkall.
\newblock {A discrete version of the Darboux transform for isothermic
  surfaces}.
\newblock In A.~I. Bobenko and R.~Seiler, editors, \emph{Discrete integrable
  geometry and physics}, pages 59--81. Oxford University Press, 1999.

\bibitem[Hirota(1977)]{Hirota:1977cj}
R.~Hirota.
\newblock {Nonlinear Partial Difference Equations III; Discrete Sine-Gordon
  Equation}.
\newblock \emph{Journal of the Physical Society of Japan}, 43\penalty0
  (6):\penalty0 2079--2086, 1977.

\bibitem[Hoffmann(1999)]{Hoffmann:1999wq}
T.~Hoffmann.
\newblock {Discrete Amsler surfaces and a discrete Painlev{\'e} III equation}.
\newblock In A.~I. Bobenko and R.~Seiler, editors, \emph{Discrete integrable
  geometry and physics}, pages 83--96. Oxford University Press, 1999.

\bibitem[Hoffmann et~al.(2014)Hoffmann, Sageman-Furnas, and
  Wardetzky]{Hoffmann:2014wq}
T.~Hoffmann, A.~O. Sageman-Furnas, and M.~Wardetzky.
\newblock {A discrete parametrized surface theory in $R^3$}.
\newblock \emph{arXiv preprint 1412.7293v1}, 2014.

\bibitem[Konopelchenko and Schief(1999)]{Konopelchenko:1999te}
B.~G. Konopelchenko and W.~K. Schief.
\newblock {Trapezoidal discrete surfaces: geometry and integrability}.
\newblock \emph{Journal of Geometry and Physics}, 31\penalty0 (2):\penalty0
  75--95, 1999.

\bibitem[Nimmo and Schief(1997)]{Nimmo:1997dg}
J.~J.~C. Nimmo and W.~K. Schief.
\newblock {Superposition principles associated with the Moutard transformation:
  an integrable discretization of a (2+1)-dimensional sine-Gordon system}.
\newblock \emph{Proceedings of the Royal Society A: Mathematical, Physical and
  Engineering Sciences}, 453\penalty0 (1957):\penalty0 255--279, 1997.

\bibitem[Pinkall(2008)]{pinkall2008designing}
U.~Pinkall.
\newblock {Designing cylinders with constant negative curvature}.
\newblock In A.~I. Bobenko, P.~Schr{\"o}der, J.~M. Sullivan, and G.~M. Ziegler,
  editors, \emph{Discrete Differential Geometry}, pages 57--66. Springer, 2008.

\bibitem[Rogers and Schief(2002)]{Rogers:2002wy}
C.~Rogers and W.~K. Schief.
\newblock \emph{{B{\"a}cklund and Darboux Transformations: Geometry and Modern
  Applications in Soliton Theory}}.
\newblock Cambridge Texts in Applied Mathematics. 2002.

\bibitem[Sauer(1950)]{Sauer:1950ca}
R.~Sauer.
\newblock {Parallelogrammgitter als Modelle pseudosph{\"a}rischer Fl{\"a}chen}.
\newblock \emph{Mathematische Zeitschrift}, 52\penalty0 (1):\penalty0 611--622,
  1950.

\bibitem[Schief(2006)]{Schief20061484}
W.~K. Schief.
\newblock {On a maximum principle for minimal surfaces and their integrable
  discrete counterparts}.
\newblock \emph{Journal of Geometry and Physics}, 56\penalty0 (9):\penalty0
  1484--1495, 2006.

\bibitem[Schief(2003)]{Schief:2003ug}
W.~K. Schief.
\newblock {On the unification of classical and novel integrable surfaces. II.
  Difference geometry}.
\newblock \emph{Proceedings of the Royal Society of London. Series A:
  Mathematical, Physical and Engineering Sciences}, 459\penalty0
  (2030):\penalty0 373--391, 2003.

\bibitem[Sym(1985)]{Sym:1985kl}
A.~Sym.
\newblock {Soliton surfaces and their applications (soliton geometry from
  spectral problems)}.
\newblock In R.~Martini, editor, \emph{Lecture Notes in Physics}, pages
  154--231. Springer Berlin Heidelberg, 1985.

\bibitem[Wunderlich(1951)]{Wunderlich:1951wc}
W.~Wunderlich.
\newblock \emph{{Zur Differenzengeometrie der Fl{\"a}chen konstanter negativer
  Kr{\"u}mmung}}.
\newblock Springer Verlag, 1951.

\end{thebibliography}

\end{document}